\documentclass{amsart}
\usepackage{amssymb,amscd,verbatim,latexsym}
\usepackage{enumerate}
\usepackage[mathscr]{eucal}

\date\today
\usepackage{color}
\newcommand\ede{ \, := \, }

\newcommand\rp{'}

\newcommand\drp{''}

\newcommand{\supp}{\operatorname{supp}}

\newcommand\pullback{\sp{\downarrow\downarrow}}
\newcommand{\tto}{\rightrightarrows}
\newcommand\mathbfPsi{\mathbf \Psi}

\newcommand{\Prim}{\operatorname{Prim}}

\newcommand{\CC}{\mathbb C}

\newcommand{\RR}{\mathbb R}

\newcommand{\ZZ}{\mathbb Z}

\newcommand\pa{{\partial}}
\newcommand\Cstar{C\sp{\ast}}

\newcommand\Cs[1]{C\sp{\ast}(#1)}
\newcommand\rCs[1]{C_r\sp{\ast}(#1)}

\newcommand\ssub{stratified submersion}

\newcommand{\maC}{\mathcal C}
\newcommand{\cC}{\mathcal C}

\newcommand{\maF}{\mathcal F}
\newcommand{\maG}{\mathcal G}
\newcommand{\cG}{\mathcal G}
\newcommand{\maH}{\mathcal H}
\newcommand{\cH}{\mathcal H}

\newcommand{\cJ}{\mathcal J}
\newcommand{\maK}{\mathcal K}
\newcommand{\maL}{\mathcal L}

\newcommand{\maP}{\mathcal P}
\newcommand{\maR}{\mathcal R}

\newcommand{\cV}{\mathcal V}
\newcommand{\maW}{\mathcal W}

\newcommand{\de}{{\rm d}}
\newcommand{\maJ}{\mathcal J}
\def\pa{\partial}
\newtheorem{theorem}{Theorem}[section]
\newtheorem{proposition}[theorem]{Proposition}
\newtheorem{corollary}[theorem]{Corollary}

\theoremstyle{definition}
\newtheorem{definition}[theorem]{Definition}
\theoremstyle{remark}
\newtheorem{remark}[theorem]{Remark}
\newtheorem{example}[theorem]{Example}

\author[C. Carvalho]{Catarina Carvalho} \address{Dep. Matem\'{a}tica,
    Instituto Superior T\'{e}cnico, University of Lisbon, Av. Rovisco
    Pais, 1049-001 Lisbon, Portugal }
\email{catarina.carvalho@math.tecnico.ulisboa.pt}

\author[Y. Qiao]{Yu Qiao} \address{School of Mathematics and
  Information Science,\\ Shaanxi Normal University, Xi'an, 710119,
  China} \email{yqiao@snnu.edu.cn}

\thanks{Carvalho was partially supported by Funda\c c\~ao para a Ci\^{e}ncia e a Tecnologia, Portugal,
UID/MAT/04721/2013. Qiao
  was partially supported by NSF of China 11301317 and the Foundational
  Research Funds for the Central Universities GK201803003.
}
\date\today

\title[Fredholm Groupoids and Layer Potentials ]{Fredholm Groupoids and Layer Potentials on Conical Domains}

\begin{document}

\begin{abstract}
We show that layer potential groupoids for conical domains constructed in
 an earlier paper (Carvalho-Qiao, Central European J. Math., 2013) are
 Fredholm groupoids, which enables us to deal with many analysis problems
 on singular spaces in a unified treatment. As an application,
 we obtain Fredholm criteria for operators on layer potential groupoids.
\end{abstract}

\maketitle

\tableofcontents

\section{Introduction}

Lie groupoids are effective tools to model analysis problems on singular spaces, for a small sample of applications see,  for instance, \cite{MonthubertSchrohe, AJ2,  ALN, DebordLescure1, DebordLescure2, DLN,  LN, Monthubert01, MonthubertNistor, MRen, Nistor16} and references therein. One general advantage behind this strategy is that, by associating a Lie groupoid to a given singular problem, not only we are able to apply groupoid techniques, but also get automatically a groupoid $C^*$-algebra and well-behaved pseudodifferential calculi naturally affiliated to
this $C^*$-algebra \cite{ASkandalis2, LMN, LN, Monthubert03, NWX,vanErpYuncken}.

In what regards Fredholm criteria, in the singular case we often obtain Fredholm conditions of the form ``$P$ is Fredholm if, and only if, $P$ is elliptic \emph{and} a family of limit operators $P_{\alpha}$ is invertible". In many situations, this family of operators can be obtained from suitable representations of the groupoid $C^{*}$-algebra, and so we can use representation theory to study Fredholmness.

Recently \cite{CNQ, CNQ17}, with Victor Nistor, the notion of \emph{Fredholm groupoid} was considered as, in some sense, the largest class of Lie groupoids for which such Fredholm criteria hold with respect to a natural class of representations, the {regular representations} (see Section \ref{s.fredholm} for the precise definitions). A characterization of such groupoids is given relying on the notions of \emph{strictly spectral} and \emph{exhaustive} families of representations, as in \cite{nistorPrudhon, Roch}.
The associated non-compact manifolds are named \emph{manifolds with amenable ends}, since certain isotropy groups at infinity are assumed to be
amenable.
This is the case for manifolds with cylindrical and poly-cylindrical ends, for
manifolds that are asymptotically Euclidean, and for manifolds that
are asymptotically hyperbolic, and also manifolds obtained by iteratively blowing-up singularities. In \cite{CNQ} we discuss these examples extensively, and show how the Fredholm groupoid approach provides an unified treatment for many singular problems.

In the present paper, our purpose is to relate the Fredholm groupoid approach to the study of \emph{ layer potential operators} on domains with \emph{conical singularities}. Our motivation comes from the study of boundary problems for elliptic equations,
namely by applications of the classical {method of layer potentials}, which reduces differential equations to  \emph{boundary} integral equations. One typically wants to invert an
operator of the form ''$\frac{1}{2} +K$" on suitable
function spaces on the boundary of some domain $\Omega$. If the boundary is $\maC^2$, or even $\cC^1$, then the
integral operator $K$ is compact \cite{FJR, Fol, Kress} on
$L^2(\pa\Omega)$, so the operator $\frac{1}{2}+K$ is Fredholm and
we can apply the classical Fredholm theory to solve the Dirichlet problem.  But if
there are singularities on the boundary, as in the case of conical domains,  this result is not necessarily true \cite{Els, FJL, Kon, Kress,
Lew, LP, IMitrea2, IMitrea3, MitreaNistor}.
 Suitable groupoid $C^{*}$-algebras, and their representation theory,  are then a means to provide the right replacement for the compact operators, and  the theory of Fredholm groupoids is suited to yield the desired Fredholm criteria.

We consider here bounded domains with conical points $\Omega$ in
$\mathbb{R}^n$, $n \ge 2$, that is, $\overline{\Omega}$ is locally diffeomorphic
to a cone with smooth, \emph{possibly disconnected}, base. (If $n=2$, we allow $\Omega$ to be a domain with cracks. See Section  \ref{s.LP_groupoids} for the precise definitions.)
In \cite{CQ13},  the authors associated
   to $\Omega$, or
more precisely to $\pa \Omega$, a \emph{layer potentials groupoid}  over the (desingularized) boundary that aimed to provide the right setting to study invertibility and Fredholm problems as above.
As a space, we have
\begin{equation*}\label{grpd.nocrack}
  \cG:= \bigsqcup\limits_{i} (\pa\omega_i \times \pa\omega_i) \times (\RR^+ )  \quad \bigsqcup \quad (\Omega_0 \times \Omega_0) \quad  \tto \quad M:=\left(\bigsqcup\limits_{i } \partial\omega_{i} \times [0,1) \right)\quad
    \bigsqcup \quad \Omega_0
\end{equation*}
where $\Omega_0$ is the smooth part of $\pa \Omega$, and the local cones have bases  $\omega_{i} \subset
  S^{n-1}$,
  with smooth boundary, , $i=1,...,l$. The space of units $M$ can be thought of as a desingularized boundary. The limit operators in this case, that is, the operators over $M\setminus \Omega_{0}$,  have dilation invariant kernels  on $(\pa\omega_i\times \pa\omega_i) \times (\RR^+ )$, that eventually yield a family of Mellin convolution operators on $( \pa\omega_i) \times (\RR^+ )$, indexed on each local cone. This fact was one of the original motivations in our definition.
In \cite{CQ13}, we were able to obtain Fredholm criteria making use of the machinery of pseudodifferential operators on Lie manifolds \cite{ALN}.

In this paper, we go further to show that the layer potentials groupoid associated to (the boundary of) a conical domain is indeed a Fredholm groupoid (Theorem \ref{thm.LPFred}).
We can then place the layer potentials approach in the framework of Fredholm groupoids. Moreover, we obtain the Fredholm criteria naturally
and extend to  a space of operators  that contains $L^{2}$-inverses.
 These Fredholm criteria are formulated on weighted Sobolev spaces,
 we refer the reader to \cite{Kon,MR} and references therein.
 We recall their definition: let $r_{\Omega}$ be the  smoothed  distance
function to the set of conical points of $\Omega$. We define the $m$-th
Sobolev space on $\pa\Omega$ with weight $ r_{\Omega}$ and index $a$ by
\begin{equation*}\label{eq.def.ws}
  \maK_{a}^m(\pa\Omega)=\{u\in L^2_{\text{loc}}(\pa\Omega), \, \,
  r_{\Omega}^{|\alpha|-a}\partial^\alpha u\in L^2(\pa\Omega), \,\,\,\text{for
    all}\,\,\, |\alpha|\leq m\}.
\end{equation*}
We have the following isomorphism \cite{BMNZ}:
$$ \maK^{m}_{\frac{n-1}{2}}(\partial\Omega)\simeq
  H^{m}(\partial'\Sigma(\Omega),g), \quad  \mbox{  for all  } m\in \mathbb{R}.$$
where $\Sigma(\Omega)$ is a desingularization, and $\partial'\Sigma(\Omega)$
is the union of the hyperfaces that are not at infinity in $\pa \Sigma(\Omega)$, which can be identified with a desingularization of $\pa \Omega$ (see Section \ref{s.LP_groupoids}).

Applying the results for Fredholm groupoids we obtain our main result (Theorems \ref{thm.fredholm} and \ref{thm.crack}). The space $L^{m}(\maG)$ is the completion of $\Psi^m(\maG)$ with respect to the operator norm on Sobolev spaces (see Section \ref{ss.ops.grpds}).

\begin{theorem}
Let $\Omega \subset \mathbb{R}^n$ be a conical domain without cracks and
$\Omega^{(0)}=\{p_1,p_2,\cdots, p_l\}$ the set of conical points, with possibly disconnected cone base $\omega_{i}\subset S^{n-1}$. Let $\maG\tto M=\pa'\Sigma(\Omega)$
be the layer potential groupoid as in Definition \ref{gpd1}. Let $P\in L^{m}(\maG)\supset \Psi^m(\maG)$ and $s \in \RR $.
We have
\begin{equation*}
  \begin{gathered}
	P : \maK_{\frac{n-1}{2}}^s(\pa\Omega) \to \maK_{\frac{n-1}{2}}^{s-m}(\pa\Omega) \mbox{ is Fredholm}
        \ \ \Leftrightarrow \ \ P \mbox{ is elliptic and all the Mellin convolution operators  }\\
	\ P_{i}:=\pi_{p_{i}}(P) : H^s(\RR^+\times \pa\omega_i;g) \to H^{s-m}(\RR^+\times \pa\omega_i;g)\,,
	        \, \mbox{ are invertible}\,,
  \end{gathered}
\end{equation*}
where the metric $g= r^{-2}_\Omega \, g_e$ with $g_e$ the Euclidean metric.
\end{theorem}

The above theorem also holds, with modifications, for polygonal domains with ramified cracks (Theorem \ref{thm.crack}).

The layer potentials groupoid constructed here is related to the so-called  $b$-groupoid (Example \ref{bgrpd}) associated to the manifold
with smooth boundary $\pa' \Sigma(\Omega)$, which induces  Melrose's $b$-calculus
\cite{MelroseAPS}. If the boundaries of the local cones bases are connected, then the two groupoids coincide (note that it is often the case that the boundaries are disconnected, for instance take $n=2$). In general, our pseudodifferential calculus
contains the compactly supported $b$-pseudodifferential operators, in that our
groupoid contains the $b$-groupoid as an open
subgroupoid. The main difference at the groupoid level is that in the usual $b$-calculus there is no interaction between the different faces at each conical point.  

In \cite{QL18}, Li and the second-named author applied the techniques of pseudodifferential operators on Lie groupoids to the method of layer potentials on plane polygons (without cracks) to obtain the invertibility of operators $I \pm K$ on suitable weighted Sobolev spaces on the boundary, where $K$ is the double layer potential operators (also called Neumann-Poincar\'{e} operators) associated to the Laplacian and the polygon. The Lie groupoids used in that paper are exactly the groupoids we constructed in \cite{CQ13}, which will be shown to be Fredholm in this paper. Moreover, the second-named author used a similar idea to make a connection between the double layer potential operators on three-dimensional wedges and (action) Lie groupoids in \cite{Qiao18}.

We expect to  be able to use our results to show that the relevant integral operators appearing in the method of layer potentials  for domains with conical points of dimension greater than or equal to $3$ are Fredholm between suitable weighted Sobolev spaces. However, for domains with cracks, the resulting layer potential operators are no longer Fredholm. These issues will be addressed in a forthcoming paper.

Let us briefly review the contents of each section. We start with reviewing the general notions relating to Lie groupoids, groupoid $C^{*}$-algebras and pseudodifferential operators on Lie groupoids (Sections \ref{ss.Liegroupoids} and \ref{ss.ops.grpds}). Then in Section \ref{ss.fredholm}, we review the definition  of Fredholm groupoids and their characterization, relying on strictly spectral and exhaustive families of representations, resulting on Fredholm criteria for operators on Fredholm groupoids.
In Section \ref{s.LP_groupoids}, we describe the construction of layer potential groupoids on conical domains and give their main properties, in the case with no cracks (Section \ref{ss.groupoidnocrack}) and in the case of polygonal domains with ramified cracks (Section \ref{ss.groupoidcracks}). Finally, in Section \ref{s.FredCond}, we show that such groupoids are Fredholm and obtain the Fredholm criteria for layer potential groupoids.

\vspace{0.2cm}
\emph{Acknowledgements:} We would like to thank the editors for the invitation and Victor Nistor for useful discussions and suggestions.
%%%%%%%%%%%%%%%%%%%%%%%%%%%%%

\vspace{0.3cm}
\section{Fredholm Groupoids}\label{s.fredholm}

We recall some basic definitions and properties of Lie groupoids and Fredholm Lie groupoids,
and refer to Renault's book \cite{renault80} for locally compact groupoids,
Mackenzie's books \cite{Mackenzie87, Mackenzie05} for Lie groupoids,
and the papers \cite{CNQ, CNQ17} for Fredholm Lie groupoids.

\vspace{0.2cm}

\subsection{ Lie groupoids and groupoid $C^{*}$-algebras}\label{ss.Liegroupoids}

We recall that a {\em small category} is a category all of whose objects
form a set. Here is a quick definition of groupoids.
\begin{definition}
A {\em groupoid} is a small category in which every morphism is invertible.
\end{definition}

More precisely, a groupoid $\maG$ consists of two sets $\maG^{(1)}, \maG^{(0)}$ together
with structural morphisms
\begin{enumerate}
\item the domain map and range map $d, r :\maG^{(1)} \rightarrow \maG^{(0)}$,
\item the product $\mu: \maG^{(2)}:=\{(g,h) \in \maG^{(1)} \times \maG^{(1)} \,| \,d(g)=r(h) \} \rightarrow \maG^{(1)}$, written $gh:=\mu(gh)$ for simplicity,
     where $\maG^{(2)}$ is called the set of composable pairs,
\item the inverse map $\iota: \maG^{(1)} \rightarrow \maG^{(1)}$, written $g^{-1}:=\iota(g)$, and
\item the inclusion or unit map $u: \maG^{(1)} \rightarrow \maG^{(0)}$,
\end{enumerate}
satisfying the following relations
\begin{enumerate}
\item $d(gh)= d(h)$, $r(gh)=r(g)$ if $(g, h) \in \maG^{(2)}$,
\item $\mu$ is associative: $(gh)k=g(hk)$ for all $(g,h), (h,k) \in \maG^{(2)}$,
\item $u$ is injective and $d(u(x))=x=r(u(x))$ for all $x\in \maG^{(0)}$,
\item $g\,u(d(g))=g$, and $u(r(g))\,g=g$ for all $g \in \maG^{(1)}$, and
\item $r(g^{-1}) = d(g)$, $d(g^{-1})=r(g)$, $g \, g^{-1}= u(r(g))$ and $g^{-1}\, g= u(d(g))$ for all $g\in \maG^{(1)}$.
\end{enumerate}

We always identify $\maG$ with $\maG^{(1)}$, denote $M:=\maG^{(0)}$, and
usually write $\maG \tto M$ for a groupoid $\maG$ with units $M$.
Let $A, B\subset M$. We denote by $\maG_A:=d^{-1}(A)$, $\maG^B:=r^{-1}(B)$,
and $\maG_A^B:=\maG_A \cap \maG^B$. We call $\maG_A^A$ the {\em reduction}
of $\maG$ to $A$. If $\maG_A^A=\maG_A=\maG^A$, then $A$ is called {\em invariant}
and $\maG_A$ is also a groupoid, called the {\em restriction} of $\maG$ to $A$.
In particular, if $x\in M$, then $\maG_x^x=d^{-1}(x) \cup r^{-1}(x)$ is called
the {\em isotropy group} at $x$.

We would like to impose a certain topology on $\maG$.
In general, a groupoid $\maG \tto M$ is said to be {\em locally compact} if $\maG$ and $M$ are locally compact spaces with $M$ Hausdorff,
all five structure maps $d, r, \mu, \iota, u$ are continuous,
and the map $d$ is surjective and open.

Note that in the general  definition only the unit space $M$ is required to be Hausdorff, and $\maG$ may be non-Hausdorff.
However, throughout the paper, all our spaces  will be Hausdorff.

\vspace{0.2cm}

In the analysis of problems on singular spaces,
it is crucial to distinguish between smooth manifolds without corners and manifolds with boundary or corners.
By a {\em smooth manifold} we shall always mean a smooth manifold {\em without corners}.
By definition, every point $p \in M$ of a {\em manifold with corners} has a coordinate neighborhood
diffeomorphic to $[0,1)^k \times (-1,1)^{n-k}$ such that the transition functions are smooth.
The number $k$ is called the {\em depth} of the point $p$.
The set of {\em inward pointing tangent vectors} $v\in T_p(M)$ defines
a closed cone denoted by $T^+_p(M)$. A smooth map $f; M_1 \rightarrow M_2$ between
two manifolds with corners is called a {\em tame submersion} provided that
$df(v)$ is an inward pointing vector of $M_2$ if and only if $v$ is an inward pointing vector of $M_1$.
Then we introduce the notion of Lie groupoids.

\begin{definition}
A {\em Lie groupoid} is a locally compact groupoid $\maG \tto M$ such that
\begin{enumerate}
\item $\maG$ and $M$ are both manifolds with corners,
\item all five structure morphisms $d, r, \mu, u$ and $\iota$ are smooth,
\item $d$ is a tame submersion of manifolds with corners.
\end{enumerate}
\end{definition}

We remark that $(3)$ implies that each fiber $\maG_x=d^{-1}(x) \subset \maG$ is a smooth manifold (without corners) \cite{CNQ, Nistor_Comm}. Moreover, $\maG$ is Hausdorff (and second countable).

We assume all our locally compact groupoids to be endowed with a fixed (right) Haar system, denoted $(\lambda_x)$, where $x$ ranges
 through the set of units. All Lie groupoids have well-defined (right) Haar systems.

\vspace{0.1cm}
 To any locally compact groupoid $\maG$ (endowed with a Haar system),
there are associated two basic $C^*$-algebras, the {\em full} and {\em
  reduced} $C^*$-algebras $\Cs{\maG}$ and $\rCs{\maG}$, whose
definitionss we  recall now.
 Let $\maC_c(\maG)$ be the space of continuous, complex valued, compactly supported functions on $\maG$, as a $*$-associative algebra, endowed with convolution on fibres and the usual involution.
 There exists a
natural algebra norm on $\maC_c(\maG)$ defined by
\begin{equation*}
  \| \varphi\|_1 \ede \max \, \Bigl\{ \, \sup_{x\in M}\int_{\maG_x} \vert
  \varphi\vert\de\lambda_x, \, \sup_{x\in M}\int_{\maG_x}\vert
  \varphi^*\vert\de\lambda_x \, \Bigr\}.
\end{equation*}
The completion of $\maC_c(\maG)$ with respect to the norm $\|\cdot
\|_1$ is denoted $L^1(\maG)$.

For any $x\in M$, the algebra $\maC_c(\maG)$ acts as a bounded operator on $L^2(\maG_x,\lambda_x)$.
Define for any $x\in M$ the
\emph{regular} representation $\pi_x\, \colon\, \maC_c(\maG) \to
\maL(L^2(\maG_x,\lambda_x))$ by
\begin{equation*}
  (\pi_x(\varphi)\psi) (g) \ede \varphi * \psi(g) \ede \int_{\maG_{d(g)}}
  \varphi(gh^{-1}) \psi(h) d\lambda_{d(g)}(h) \,, \quad \varphi \in
  \maC_c(\maG) \,.
\end{equation*}
We have $\|\pi_x(\varphi)\|_{L^2(\maG_x)} \leq\|\varphi \|_{L^1(\maG)}$.

\begin{definition}\label{def.regular}
We  define the \emph{reduced $C\sp{\ast}$-algebra}
$C\sp{\ast}_{r}(\maG)$ as the completion of $\maC_c(\maG)$ with
respect to the norm
\begin{equation*}
  \| \varphi\|_r \ede \sup\limits_{x \in M}\|\pi_x(\varphi)\| \,
\end{equation*}
The \emph{full $C\sp{\ast}$-algebra} associated to $\maG$, denoted
$C\sp{\ast}(\maG)$, is defined as the completion of $\maC_c(\maG)$
with respect to the norm
\begin{equation*}
  \| \varphi\| \ede \sup\limits_\pi\|\pi(\varphi)\| \,,
\end{equation*}
where $\pi$ ranges over all {\em contractive} $*$-representations of
$\maC_c(\maG)$, that is, such that $\| \pi(\varphi)\|\leq \| \varphi \|_1$, for all $\varphi \in \maC_c(\maG)$.

The groupoid~$\maG$ is said to be \emph{metrically amenable} if the
canonical surjective $*$-homomorphism $\Cstar(\maG) \to
\Cstar_{r}(\maG)$, induced by the definitions above, is also
injective.

\end{definition}

Let $\maG \tto M$ be a second countable, locally compact groupoid with
a Haar system. Let $U \subset M$ be an open $\maG$-invariant subset,
$F := M \smallsetminus U$.
Then, by the classic results of \cite{MRW87, MRW96, renault91},  $C\sp{\ast}(\maG_U)$ is a closed
  two-sided ideal of $C\sp{\ast}(\maG)$ that yields the short exact
  sequence
\begin{equation}\label{renault.exact_item1}
  0\to C\sp{\ast}(\maG_U) \to
  C\sp{\ast}(\maG)\mathop{\longrightarrow}\limits^{\rho_F}
  C\sp{\ast}(\maG_{ F })\to 0 \,,
\end{equation}
where $\rho_{F}$ is the (extended) restriction map.
 If $\maG_F$ is metrically
  a\-me\-nable, then one also has the exact sequence
\begin{equation}\label{renault.exact_item3}
  0\to C\sp{\ast}_{r}(\maG_U)\to C\sp{\ast}_{r}(\maG)
  \mathop{\xrightarrow{\hspace*{1cm}}}\limits^{(\rho_F)_{r}}
  C\sp{\ast}_{r}(\maG_F)\to 0 \,.
\end{equation}

It follows from the Five Lemma that if the groupoids $\maG_F$ and $\maG_U$ (respectively, $\maG$)
  are metrically amenable, then $\maG$ (respectively, $\maG_U$) is
  also metrically amenable.
We notice that these exact sequences correspond to a disjoint union decomposition $\maG = \maG_F \sqcup
\maG_U.$

\vspace{0.2cm}
\subsection{Pseudodifferential operators on Lie groupoids}\label{ss.ops.grpds}
We recall in this subsection the construction of pseudodifferential operators on
Lie groupoids \cite{LMN, LN, Monthubert01, Monthubert03, MonthubertPierrot, NWX}.
Let $P=(P_x)$, $x\in M$ be a smooth family of
pseudodifferential operators acting on $\maG_x:=d^{-1}(x)$. The family
$P$ is called \emph{right invariant} if $P_{r(g)}U_g=U_gP_{d(g)}$, for all
$g\in \maG$, where
\begin{equation*}
  U_g : \maC^\infty(\maG_{d(g)}) \rightarrow
  \maC^\infty(\maG_{r(g)}), \,\ (U_gf)(g')=f(g'g).
\end{equation*}
Let $k_x$ be the distributional kernel of $P_x$, $x\in M$. Note that
the support of the $P$
\begin{equation*}
  \text{supp}(P):= \overline{\bigcup_{x\in M}\text{supp}(k_x)}  \subset \{(g,g'), \ d(g)=d(g')\} \subset \maG \times \maG
\end{equation*}
  since $\text{supp}(k_x)\subset
  \maG_x\times\maG_x$. Let $\mu_1(g',g) :=
  g'g^{-1}$. The family $P = (P_x)$ is called \emph{uniformly
    supported} if its \emph{reduced support} $\text{supp}_\mu(P) :=
  \mu_1(\text{supp}(P))$ is a compact subset of $\maG$.

\begin{definition}\label{def.C*}
 The space $\Psi^{m}(\maG)$ of \emph{pseudodifferential operators of
   order $m$ on a Lie groupoid} $\maG$ with units $M$ consists
 of smooth families of pseudodifferential operators $P=(P_x)$, $x\in
 M$, with $P_x\in \Psi^m(\maG_x)$, which are {uniformly
   supported} and {right invariant}.
\end{definition}

We also denote $\Psi^\infty(\maG) := \bigcup_{m\in
  \mathbb{R}}\Psi^m(\maG)$ and $\Psi^{-\infty}(\maG) :=
\bigcap_{m\in \mathbb{R}}\Psi^m(\maG)$. We then have a
representation $\pi_0$ of $\Psi^{\infty}(\maG)$ on $\maC^\infty_c(M)$ (or
on $\maC^\infty(M)$, on $L^2(M)$, or on Sobolev spaces), called
the \emph{vector representation} uniquely determined by the equation
\begin{equation}\label{vector_repn}
  (\pi_0(P)f)\circ r := P(f\circ r),
\end{equation}
where $f\in \maC^\infty_c(M)$ and $P=(P_x)\in \Psi^m(\maG)$.  For \emph{Hausdorff} groupoids, which is the case here, by results of Koshkam and Skandalis \cite{KSkandalis}, $\pi_{0}$ is always injective, so elements of $\Psi^{\infty}(\maG)$ can be identified with operators on $M$.

If $k_x$ denotes the distributional kernel of $P_x$, $x\in M$,
then the formula
$$k_P(g):=k_{d(g)}(g, d(g))$$
defines a distribution on the groupoid $\maG$,
with $\text{supp}( k_p) = \text{supp}_\mu(P)$ compact, smooth
outside $M$ and given by an oscillatory integral on a neighborhood of $M$.
If $P\in \Psi^{-\infty}(\maG)$, then $P$ identifies with a convolution
operator with kernel a
smooth, compactly supported function and $\Psi^{-\infty}(\maG)$
identifies with the smooth convolution algebra $\maC_c^\infty(\maG)$.  In
particular, we can define
\begin{equation*}
  \|P\|_{L^1(\maG)} := \sup\limits_{x\in M} \Big\{ \
  \int_{\maG_x}|k_P(g^{-1})|\, d\mu_x(g),\,\, \int_{\maG_x}|k_P (g)|\,
  d\mu_x(g)\ \Big\}.
\end{equation*}

For each $x\in M$, the {\em regular representation} $\pi_x$ extends to $\Psi^\infty(\maG)$,
defined by $\pi_x(P)=P_x$. It is clear that if $P \in \Psi^{-n-1}(\maG)$
$$\|\pi_x(P)\|_{L^2(\maG_x)} \leq\|P\|_{L^1(\maG)}.$$
The {\em reduced $C^*$--norm} and the {\em full norm} of $P$ are defined by
\begin{equation*}
  \|P\|_r = \sup\limits_{x\in M}\|\pi_x(P)\| = \sup\limits_{x\in
    M}\|P_x\|, \quad \quad \mbox{ and } \quad \|P\| = \sup\limits_{\rho}\|\rho(P)\|,
\end{equation*}
where $\rho$ ranges over all bounded representations of $\Psi^0(\maG)$
satisfying
\begin{equation*}
  \|\rho(P)\| \leqslant \|P\|_{L^1(\maG)}\quad \text{for all} \quad  P\in
  \Psi^{-\infty}(\maG).
\end{equation*}
We obtain $C^*(\maG)$, respectively, $C^*_r(\maG)$, to be the
completion of $\Psi^{-\infty}(\maG)$ in the norm $\|\cdot\|$,
respectively, $\|\cdot\|_r$.

Since the algebras $\Psi^m(\maG)$ are too small to contain resolvents,
we consider its $L^{m}_{s}(\maG)$ completion with respect to the norm
\begin{equation*}\label{eq.def.ms}
  \|P\|_{m, s} \ede \|(1 + \Delta)\sp{(s-m)/2} P
  (1 + \Delta)\sp{-s/2} \|_{L^2 \to L^2}.
\end{equation*}
The space $L^{m}_{s}(\maG)$
is the {\em norm closure} of $\Psi^m(\maG)$ in the topology of
continuous operators $H^s(M )\to H^{s-m}(M)$, where as usual, $H^{s}(M)$ is the domain of $(1 + \Delta)^{s/2}$, if $s \ge 0$,
whenever $M$ is compact (see
\cite{AIN, Grosse.Schneider.2013, LN}).
 Moreover, let
\begin{equation*}
 \maW^{m}(\maG) \ede \Psi^{m}(\maG) + \cap_{s} L^{-\infty}_{s}(\maG)\,.
\end{equation*}
Then $\maW^{m}(\maG) \subset L^{m}_{s}(\maG)$ and $\maW^{\infty}(\maG)$
is an algebra of pseudodifferential operators that contains the inverses
of its $L^2$-invertible operators.

\vspace{0.2cm}
Let us give some examples of Lie groupoids that will have a role in our constructions.

\begin{example}[Bundles of Lie groups]
Any Lie group $G$ can be regarded as a Lie groupoid $\maG=G$ with exactly one unit $M= \{e\}$,
the identity element of $G$.
We have $\Psi^m(\maG) \simeq \Psi^m_{\text{prop}}(G)^G$,
the algebra of right translation invariant and properly supported pseudodifferential operators on $G$.

More generally, we can let $\maG \tto B$ be a locally trivial bundle of groups, with $d = r$,
 with fiber a locally compact group $G$.
 It is metrically amenable if, and only if,
 the group $G$ is
 amenable.
\end{example}

\begin{example}[Pair groupoids]\label{pair_gpd}
Let $M$ be a smooth manifold (without corners). Then the {\em pair groupoid} $\maG := M \times M$ of $M$ is the
groupoid having {\em exactly} one arrow between any two units.
In this case, we have $\Psi^m(\maG) \simeq \Psi^m_{\text{comp}}(M)$,
the algebra of compactly supported pseudodifferential operators on $M$. For any $x\in M$, the regular
representation $\pi_x$ defines an isomorphism between
$C\sp{\ast}(M\times M)$ and the ideal of compact operators in
$\maL(L^2(M))$. In particular, all pair groupoids are metrically
amenable.
\end{example}

\begin{example}[Fibered pull-back groupoids]\label{fib_pair_gpd}
Let
$\maH \tto B$ be a groupoid and $f : M \to B$ be continuous. An important
generalization of the pair groupoid is the {\em fibered pull-back
  groupoid}:
\begin{equation*}
  f\pullback (\maH) \ede \{\, (m, g, m\rp) \in M \times \maH \times M,
  \ f(m) = r(g),\, d(g) = f(m\rp) \, \} \,,
\end{equation*}
with units $M$ and $(m, g, m\rp) (m\rp, g\rp, m\drp) = (m, g
g\rp, m\drp)$. It is a subgroupoid of the product of the pair groupoid $M\times M$ and $\maH$.
If $\maH$ is a Lie groupoid and $f$ is a tame submersion, then $f\pullback
 (\maH)$ is a Lie groupoid.

Let $\maH \tto B$ be a locally trivial bundle of groups (so $d = r$)
 with fiber a locally compact group $G$.  Also, let $f : M \to
 B$ be a continuous map that is a local fibration. Then $f\pullback
 (\maH)$ is a locally compact groupoid with a Haar system. If $G$ is a Lie group, $M$ is a manifold with corners and $f$ is a tame submersion, then $f\pullback
 (\maH)$ is a Lie groupoid.
Again,  it is metrically amenable if, and only if,
 the group $G$ is
 amenable.
\end{example}

\begin{example}[Disjoint unions]\label{ex.help-for-lp}
Let  $M$ be a {smooth } manifold and let $\maP=\{M_{i}\}_{i=1}^{p}$ be a \emph{finite} partition of $M$ into smooth disjoint, closed  submanifolds $M_{i}\subset M$ (since $\maP$ is finite, $M_{i}$ is also open, $i=1,..., p$, and the sets $M_{i}$ are always given by unions of connected components of $M$).
Let $f: M \to \maP$, $x \mapsto M_{i}$, with $x\in M_{i}$, be the quotient map .
Then $\maP$ is discrete and $f$ is locally constant,
so any Lie groupoid $\maH \tto \maP$ yields a Lie groupoid $f\pullback(\maH)\tto M$.
In particular, if $\maH=\maP$ as a (smooth, discrete) manifold, then
$f\pullback(\maP)$ is the topological disjoint union
$$f\pullback(\maP)= \bigsqcup_{i=1}^{p}(M_{i} \times M_{i}).$$
Let $G$ be a Lie group and $\maH := B \times
G$, the product of a manifold and a Lie group, then
$$f\pullback(\maH)= \bigsqcup_{i}^{p}(M_{i} \times M_{i}) \times G.$$
\end{example}

\begin{example}[Transformation groupoid]\label{transformation}
If $G$ is a Lie group acting smoothly from the right on a manifold $M$,
the associated transformation groupoid is defined to be $\maG:=M \rtimes G \tto M\times\{e\} \simeq M$
in which $\maG^{(1)}= M \times G$, with
the domain map $d(m, g)=m\cdot g$ and the range map $r(m,g)= m$,

One case of interest here is when $\maG:=[0,\infty) \rtimes (0,\infty) $ is the transformation groupoid
with the action of $(0,\infty)$ on $[0,\infty)$ by dilation. Then the $C^*$-algebra associated to $\cG$ is the algebra of Wiener-Hopf
operators on $\RR^+$, and its unitalization is the algebra of Toeplitz
operators  \cite{MRen}.
\end{example}

\begin{example}[$b$-groupoid]\label{bgrpd}

Let $M$ be a manifold with smooth boundary and let $\cV_b$ denote the
class of vector fields on $M$ that are tangent to the boundary.
The associated groupoid was
defined in \cite{MelroseAPS, Monthubert03, NWX}.
Let
$$
    \cG_b:=\Big( \bigcup\limits_{j} \RR^+\times (\pa_j M)^2\Big)
    \quad \cup \quad M_0^2,
$$
where $M_0^2$ denotes the pair groupoid of $M_0:= int(M)$ and $\pa_j
M$ denote the connected components of $\pa M$. Then $\cG_b$ can be
given the structure of a Lie groupoid with units $M$, given locally by a transformation groupoid.
It integrates the so-called $b$-tangent bundle ${^bTM}$, that is, $A(\cG_b)= {^bTM}$, the Lie algebroid whose space of sections is given by vector fields tangent to the boundary. The
pseudodifferential calculus obtained is Melrose's small $b$-calculus with compact supports. See
\cite{MelroseAPS, Monthubert03, MonthubertPierrot, NWX} for details.
\end{example}

%%%%%%%%%%%%%%%%%%%%%%%%%%%%%%%%%%%%%%%%%%%%%%%%%%
%%%%%%%%%%%%%%%%%%%%%%%%%%%%%%%%%%%%%%%%%%%%%%%%%%%

\smallskip
\subsection{Fredholm groupoids}\label{ss.fredholm}
Fredholm groupoids were introduced in \cite{CNQ, CNQ17} as groupoids for which an operator is Fredholm if,
and only if, its principal symbol and all its boundary restrictions
are invertible, in a sense to be made precise. We
review their definition and properties in this subsection.
\vspace{0.2cm}

Let $\maG \tto M$ be a Lie groupoid with $M$ compact, and assume that $U \subset M$
is an open, $\maG$-invariant subset such that $\maG_U \simeq U \times U$ (the pair groupoid, see Example \ref{pair_gpd}).
Let $\pi_0$ be the vector representation of $\Psi^\infty(\maG)$ on $\maC^\infty(M)$ uniquely
determined by Equation (\ref{vector_repn}). For any $x \in U$,
the regular representation $$\pi_x: \Psi^\infty(\maG) \rightarrow \text{End}(\maC_c^\infty(\maG_x))$$
is equivalent to $\pi_0$ via the range map $r: \maG_{x} \rightarrow U$  which defines a bijection.  Moreover, for Hausdorff groupoids,  the vector representation $\pi_0 \colon C_r\sp{\ast}(\maG) \to
  \maL(L^2(U))$ is injective and defines an isomorphism
  $C_r\sp{\ast}(\maG_U) \simeq \pi_0(C_r\sp{\ast}(\maG_U)) = \maK$, the
  algebra of compact operators on $L\sp{2}(U)$. We
 identify $C^{\ast}_r(\maG)$ with its image under $\pi_0$,
that is, with a class of operators on $L^2(U)$, without further
comment.

\begin{definition}
A Lie groupoid $\maG \tto M$ is called a {\em Fredholm Lie groupoid} provided that
\begin{enumerate}
\item There exists an open, dense, $\maG$-invariant subset $U\subset M$ such that $\maG_U \simeq U \times U$.
\item For any $a\in C^*_r(\maG)$, we have that $1+a$ is Fredholm if,
and only if, all $1+\pi_x(a)$, $x\in F:=M\backslash U$ are invertible.
\end{enumerate}
\end{definition}

A simple observation is that $F:=M \backslash U$ is closed and $\maG$-invariant
since $U$ is a dense open set and hence completely determined by $\maG$. We shall keep this notation throughout the paper. Note also that  two regular representations $\pi_x$ and
$\pi_y$ are unitarily equivalent if , and only if, there is $g \in \maG$ such that
$d(g) = x$ and $r(g) = y$, that is, if $x, y$ are in the same orbit (of
$\maG$ acting on $M$). In particular, one only needs to verify (2) for a representative of each orbit of $\maG_{F}$.

In \cite{CNQ, CNQ17}, we gave easier to check conditions for a groupoid $\maG$ with an open, dense, subset $U$ as above, to be Freholm, depending on properties of  representations of $C_r\sp{\ast}(\maG)$. We review briefly the main notions, see \cite{nistorPrudhon, Roch} for details.

Let $A$ be a $C^*$-algebra. Recall that a two-sided ideal
$I \subset A$ is said to be {\em primitive} if it is the kernel of an irreducible
representation of $A$. We denote by $\Prim(A)$ the set of primitive ideals of $A$
and we equip it with the hull-kernel topology {(see \cite{DavidsonBook, williamsBook} for more details)}.
Let $\phi$ be a representation of $A$. The {\em support} $\supp(\phi)\subset \Prim(A)$
is defined to be the set of primitive
ideals of $A$ that contain $\ker(\phi)$.
Then in \cite{nistorPrudhon} a set of $\maF$ of representations of a
$C^*$-algebra $A$ is said to be {\em exhaustive} if $\Prim(A)= \bigcup_{\phi\in \maF} \supp(\phi)$,
that is, if any irreducible representation is weakly contained in some $\phi\in\maF$.

If $A$ is unital, then a set $\maF$ of representations of $A$ is called {\em
  strictly spectral} if  it characterizes invertibility in $A$, in that $a \in A$ is invertible if, and only if, $\phi(a)$ is invertible
for all $\phi \in \maF$. If $A$ does not have a unit, we replace
$A$ with $A\sp{+} := A \oplus \CC$ and $\maF$ with $\maF\sp{+} := \maF
\cup \{\chi_0 : A\sp{+} \to \CC\}$,
where $\maF$ is regarded as a family of representations of $A^+$.
Note that strictly spectral families of representations consist
of non-degenerate representations, and any non-degenerate representation of
a (closed, two-sided) ideal in a $C\sp{\ast}$-algebra always has a unique
extension to the whole algebra \cite{nistorPrudhon}.

It was proved in \cite{nistorPrudhon, Roch} that, if $\maF$ is exhaustive, then $\maF$ is strictly
spectral, and the converse also holds  if $A$ is separable.

The next result was given  in \cite{CNQ, CNQ17} and  gives a characterization of Fredholm
groupoids. For a groupoid $\maG$, we usually denote by
 $\maR(\maG)$ the set of its regular representations.

\begin{theorem}  \label{thm.Fredholm.Cond}
Let $\maG \tto M$ be a Lie groupoid and $U$ an open, dense, $\maG$-invariant subset  such that $\maG_U \simeq U \times U$, $F=M \backslash U$. If $\maG$ is a Fredholm groupoid, we have:
\begin{enumerate}[(i)]

\item The canonical projection induces an isomorphism
  $C_r\sp{\ast}(\maG)/C_r\sp{\ast}(\maG_{U}) \simeq
  C_r\sp{\ast}(\maG_F)$, that is, we have the exact sequence
  \begin{equation*}
  0  \longrightarrow C\sp{\ast}_{r}(\maG_U)\cong \maK \longrightarrow C\sp{\ast}_{r}(\maG)
  \mathop{\xrightarrow{\hspace*{1cm}}}\limits^{(\rho_F)_{r}}
  C\sp{\ast}_{r}(\maG_F)  \longrightarrow 0 \,.
\end{equation*}
\item $\maR(\maG_{F})=\{\pi_x,\, x \in F\}$ is a strictly spectral, or equivalently, an exhaustive, set of
  representations of $C_r\sp{\ast}(\maG_F)$.
\end{enumerate}

Conversely, if $\maG \tto M$ satisfies (i) and (ii), then, for any
unital $C^{\ast}$-algebra $\mathbfPsi$ containing $C^{\ast}_r(\maG)$
as an essential ideal,
and for any $a \in \mathbfPsi $, we have that
$a$ is Fredholm on $L^2(U)$ if, and only if, $\pi_x(a)$ is invertible for
each $x \notin U$ {\bf and} the image of $a$ in
$\mathbfPsi/C^{\ast}_r(\maG)$ is invertible.
\end{theorem}

In \cite{CNQ, CNQ17}, we dubbed condition (ii) as {\em Exel's property} (for  $\maG_F$). If $\maR(\maG_{F})=\{\pi_x,\, x \in F\}$ is a strictly spectral/ exhaustive, set of
  representations of $C\sp{\ast}(\maG_F)$, then $\maG_{F}$ is said to have {\em Exel's strong property}. In this case, it is metrically amenable.
We will use the sufficient conditions in Theorem \ref{thm.Fredholm.Cond} in the following form:

\begin{proposition}\label{cor.Fredholm.Cond}
Let $\maG \tto M$ be a Lie groupoid and $U$ an open, dense, $\maG$-invariant subset  such that $\maG_U \simeq U \times U$, $F=M \backslash U$.  Assume $\maR(\maG_{F})=\{\pi_x,\, x \in F\}$ is a strictly spectral, or equivalently, an exhaustive, set of
  representations of $C\sp{\ast}(\maG_F)$. Then $\maG$ is Fredholm and metrically amenable.
  \end{proposition}

\begin{proof}
Condition (ii)  in Theorem \ref{thm.Fredholm.Cond}  holds by assumption.
If $\maR(\maG_{F})$ is a strictly spectral set of
  representations of $C\sp{\ast}(\maG_F)$ then, by definition, the reduced and full norms coincide, hence $\maG_{F}$  is metrically amenable. It follows from the exact sequences  \eqref{renault.exact_item1} and \eqref{renault.exact_item3}, since $\maG_{U} \simeq U \times U$ is metrically amenable, that $\maG$ is metrically amenable and that condition (i) in Theorem \ref{thm.Fredholm.Cond} also holds.
   Taking the unitalization $\mathbfPsi:=\left(C^{\ast}(\maG)\right)^{+}$, we have then that $\maG$ is Fredholm.
\end{proof}

Representations are extended to matrix algebras in the
obvious way, which allows us to treat operators on vector bundles.

\begin{remark}
The notion of exhaustive family can be linked to that of $EH$-amenability and to the Effros-Hahn conjecture \cite{CNQ17, nistorPrudhon}.
Let $\maG\tto F$ be an $EH$-amenable locally compact groupoid.  Then
the family of regular representations $\{\pi_y, y \in F\}$ of $C\sp{\ast}(\maG)$
is exhaustive, hence strictly spectral.
Hence if $U$ is a dense invariant subset such that $\maG_{U}$ is the pair groupoid and $\maG_{F}$ is $EH$-amenable, then $\maG$ is Fredholm. Combining with the proof of the generalized EH conjecture
\cite{ionescuWilliamsEHC, renault87, renault91} for amenable,
Hausdorff, second countable groupoids, we get a set of sufficient conditions for $\maG$ to be Fredholm.
\end{remark}

\begin{example}\label{expl.transfgroupoidFredholm}
Let $\overline{\maH}=[0,\infty] \rtimes (0,\infty)$ be the transformation groupoid
with the action of $(0,\infty)$ on $[0,\infty ]$ by dilation,
(that is, $\overline{\maH}$ is the extension of the groupoid in Example \ref{transformation} to the one point compactification of $[0,\infty)$).
Then $\overline{\maH}$ is Fredholm.

It is clear that $(0,\infty)\subset [0,\infty]$ is an invariant open dense subset, and $\overline{\maH}|_{(0,\infty)} \simeq (0,\infty)^2$,
the pair groupoid of $(0,\infty)$.
Then $F=\{0, \infty\}$, $\overline{\maH}_{F}\cong (0,\infty)\sqcup (0,\infty)$, the disjoint union of two amenable Lie groups, and $C\sp{\ast}(\overline{\maH}_{F})\cong \maC_{0}(\RR^{+}) \oplus \maC_{0}(\RR^{+})$. Hence $\overline{\maH}_{F}$ has Exel's property (the regular representations at $0$ and $\infty$ are induced from the regular representation of the group, which is just convolution).

Note that if we have a convolution operator $K$ on the abellian group $(0,\infty)$, for instance the double layer potential operator, we can identify $K$ with a family of convolution operators $K_x$, $x\in (0,\infty)$ (we use the fact that the action groupoid $(0,\infty) \rtimes (0,\infty)$ is isomorphic to the pair groupoid of $(0,\infty)$.)
Since each $K_x$ is a convolution operator, we can
always extend by continuity the family $K_x$, $x\in (0,\infty)$ to the family $K_x$, $x\in [0,\infty]$ (two endpoints included).
In this way, we identify $K$ with an operator on the groupoid $[0,\infty] \rtimes (0,\infty)$ (note however, that the reduced support of $K$ may not be compact, so  it might not be a pseudodifferential operator on the groupoid $\overline{\maH}$, according to our previous definition).

\end{example}

The characterization of Fredholm groupoids given above, together with the properties of exhaustive / strictly spectral families, allows us to show that large classes of groupoids are Fredholm. For instance, it is easy to see that the product of Fredholm groupoids $\maG_1 \times \maG_2 \tto M_1 \times M_2$ is also Fredholm, as the regular representations of the product are direct sums of the regular  representations of $\maG_{1}$ and $\maG_{2}$. Moreover, the gluing of Fredholm groupoids along an open subset of $U$ is also Fredholm. One important observation is that Fredholmness is preserved under groupoid equivalence, as this relation  yields Morita equivalence of the groupoid $C^{*}$-algebras and isomorphic primitive ideals spaces, by the classical results of \cite{MRW87,rieffelInducedCstar, SW12,Tu04}.

In the next example, we see an important class of Lie
groupoids  for which the set  of regular
representations is an exhaustive set of representations of
$\Cs{\maG}$. The point is that
 locally, our groupoid is the product of a group $G$ and a space, so its
$C^*$-algebra  is of the form $C^*(G) \otimes \maK$, where
$\maK$ are the compact operators.
See \cite{CNQ} (Proposition 3.10) for a complete proof.

\begin{example}\label{ex.Exel}
 Let $\maH \tto B$ be a locally trivial bundle of groups, so $d = r$,
 with fiber a locally compact group $G$. Then $\maH$ has Exel's property, that is, the set of regular representations $\maR(\maH)$ is exhaustive / strictly spectral for $C^{\ast}_r(\maH)$, since any irreducible representation of $C_{r}^{\ast}(\maH)$ factors through evaluation  at $\maH_{x}=G$, and
the regular representations of $\maH$  are obtained from the regular representation of $G$.
  It is exhaustive for the full algebra $C^{\ast}(\maH)$ if, and only if,
 the group $G$ is
 amenable.

 More generally, let $f : M \to
 B$ be a continuous map. Then $\maG=f\pullback
 (\maH)$ is a locally compact groupoid with a Haar system that also
 has Exel's property, and $\maR(\maG)$ is exhaustive/ strictly spectral for $C^{\ast}(\maG)$ if, and only if,
 the group $G$ is
 amenable. (Note that $G$ coincides with the isotropy group $\maH_{x}^{x}$, for $x\in M$.)
  Moreover, we have that  $\maH$ and $\maG=f\pullback
 (\maH)$ are equivalent groupoids.

\begin{remark}
In fact,  $f\pullback (\maH)$
 satisfies the generalized EH conjecture, and hence it has the
 weak-inclusion property. It will be EH-amenable if, and only if, the group $G$ is
 amenable (see \cite{CNQ17}).
 \end{remark}

\end{example}

Putting together the previous example and Proposition \ref{cor.Fredholm.Cond}, we conclude the following:
\begin{corollary}\label{prop.fred}
If $\maG\tto M$ is a Lie groupoid, $U\subset M$ is an open, dense, invariant subset, $F=M\setminus U$, and we have a decomposition
\begin{equation*}
\maG = (U\times U) \sqcup f\pullback
 (\maH)
\end{equation*}
where $f: F\to B$ is a { tame submersion }and $\maH\tto B$ is  a bundle of  amenable Lie groups, then $\maG$ is Fredholm.
\end{corollary}

More generally, if we have a filtration of $M$ and our groupoid is given by fibered pull-backs on the strata then it will still be Fredholm - what we call in \cite{CNQ} \emph{\ssub\ groupoids}.
Several examples of Fredholm groupoids can be found in \cite{CNQ} (Section 5, see also \cite{CNQ17}), and include the $b$-groupoid modelling manifolds with poly-cylindrical ends, groupoids modelling analysis on asymptotically Euclidean space, asymptotically hyperbolic
space, and the edge groupoids  (in fact, these are all \ssub\ groupoids).

\smallskip

We consider Fredholm groupoids because of their applications
to Fredholm conditions. Let  $\Psi\sp{m}(\maG)$ be the space of order $m$,
classical pseudodifferential operators $P = (P_x)_{x \in M}$ on $\maG$, as in the previous subsection.
Then each $P_x \in
\Psi\sp{m}(\maG_x)$, $x \in M$ and $P_x=\pi_x(P)$, for the regular
representation $\pi_x, x\in M$.  Also, $P$ acts on $U$ via $P_{x_0}:
H^s(U) \to H^{s-m}(U)$, $x_0 \in U$.
Let $L^{m}_{s}(\maG)$
be the { norm closure} of $\Psi^m(\maG)$ in the topology of
continuous operators $H^s(M )\to H^{s-m}(M)$.
The following result can be found in \cite[Theorem 4.17]{CNQ}.

\begin{theorem}\label{thm.nonclassical2}
Let $\maG \tto M$ be a Fredholm Lie groupoid and let
$U \subset M$ be the dense, $\maG$-invariant subset such that
$\maG_{U} \simeq U \times U$.
 Let $s\in \RR$ and $P \in L^m_s(\maG) \supset \Psi\sp{m}(\maG)$.
We have
\begin{multline*} %\label{eq.Fredholm.c2}
	P : H^s(U) \to H^{s-m}(U)\
        \mbox{ is Fredholm} \ \ \Leftrightarrow \ \ P \mbox{ is
          elliptic and }\\
	\ P_{x} : H^s(\maG_x) \to H^{s-m}(\maG_x) \ \mbox{ is
          invertible for all } x \in F := M \smallsetminus U \,.
\end{multline*}
\end{theorem}

This theorem is proved by considering $a := (1 + \Delta)\sp{(s-m)/2} P
(1 + \Delta)\sp{-s/2}$, which belongs to $\overline{\Psi}(\maG) \, =:\, L^{0}_{0}(\maG)$,
by the results in \cite{LMN, LN}, and applying Theorem \ref{thm.Fredholm.Cond}.

%%%%%%%%%%%%%%%%%%%%%

\vspace{0.3cm}
\section{Layer Potentials Groupoids}\label{s.LP_groupoids}

In this section, we   review the construction of layer potentials groupoids for conical domains in \cite{CQ13}. In order to study layer potentials operators, which are operators on the boundary, we consider a groupoid over the desingularized boundary and apply the results in the previous section.

\subsection{Conical domains and desingularization}

We begin with the definition of domains with conical points \cite{BMNZ, CQ13, MN10}.

\begin{definition}\label{domain}
Let $\Omega \subset \RR^n$, $n\geqslant 2$, be an open connected
bounded domain. We say that $\Omega$ is a \emph{domain with conical
  points} if there exists a finite number of points $\{p_1, p_2,
\cdots, p_l\} \subset \partial \Omega$, such that
\begin{enumerate}
\item[(1)] $\partial \Omega \backslash \{p_1, p_2, \cdots, p_l\}$ is smooth;
\item[(2)] for each point $p_i$, there exist a neighborhood $V_{p_i}$
  of ${p_i}$, a possibly disconnected domain $\omega_{p_i} \subset
  S^{n-1}$, $\omega_{p_i} \neq S^{n-1}$, with smooth boundary, and a
  diffeomorphism $\phi_{p_i}: V_{p_i} \rightarrow B^{n}$ such that
$$\phi_{p_i}(\Omega \cap V_{p_i})=\{rx': 0<r< 1, x'\in \omega_{p_i}\}.$$
(We assume always that $\overline{V_i}\cap \overline{V_j}=\emptyset$, for $i\neq j$, $i,j\in\{1,2,\cdots, l\}$.)
\end{enumerate}
If $\pa\Omega= \pa\overline{\Omega}$, then we say that
$\Omega$ is a \emph{domain with no cracks}. The points $p_i$, $i= 1,
\cdots, l$ are called \emph{conical points} or \emph{vertices}. If
$n=2$, $\Omega$ is said to be a {\em polygonal domain}.
\end{definition}
% and that, up to a local change of
%coordinates, $  J\phi_k(0)=I_n,$ where $J\phi(0)$ is the Jacobian matrix of $\phi_i$ at $p_i$ .)
\vspace{0.2cm}

We shall distinguish two cases: \emph{conical domains without cracks}, $n\in\mathbb{N}$, and
\emph{polygonal domains with ramified cracks}. (Note that if $n\geq 3$ then domains with cracks have edges, and are no longer conical.)

\vspace{0.2cm}

For simplicity, we assume $\Omega$ to be a subset of $\RR^n$. In general, $\overline{\Omega}$ is a compact manifold with corners, with boundary points of maximum depth $2$, and all our constructions apply.

\vspace{0.2cm}

  In applications to boundary value
problems in $\Omega$, it is often useful to regard smooth boundary points as artificial
vertices, representing for instance a change in boundary conditions. Then a conical point $x$ is a smooth boundary point
if, and only if, $\omega_x\cong S_+^{n-1}$.
%so that on faces we always have the same boundary condition
The minimum set of conical points is unique and
coincides with the singularities of $\pa \Omega$; these are \emph{true
  conical points} of $\Omega$. Here we will give our results for true vertices, but the constructions can easily be extended to artificial ones.

  For the remainder of the paper, we keep the notation as in Definition \ref{domain}. Moreover,
for a conical domain $\Omega$, we always denote by $$\Omega^{(0)}=\{ p_1, p_2, \cdots, p_l\},$$
the set of true conical points of $\Omega$, and by $\Omega_0$ be the smooth part of $\pa\Omega$,
i.e., $\Omega_0=\pa\Omega \backslash \{p_1,p_2,\cdots, p_l\}$.
We remark that we allow the bases $\omega_{p_{i}}$ and $\pa \omega_{p_{i}}$ to be disconnected (in fact, if $n=2$, $\pa \omega_{p_{i}}$ is always disconnected).

\vspace{0.2cm}

We now recall the definition  of the
\emph{desingularization} $\Sigma(\Omega)$ of $\Omega$ of a conical domain, which is obtained from $\Omega$ by removing a,
possibly non-connected, neighborhood of the singular points and
replacing each connected component by a cylinder.
 We refer to  \cite{BMNZ} for details on this construction, see also \cite{CQ13, Kon, MelroseAPS}. \smallskip
We have  the following
\begin{equation*}\label{desing}
   \Sigma(\Omega) \cong \left(\bigsqcup\limits_{p_i \in
     \Omega^{(0)}}[0,1)\times \overline{\omega_{p_i}}
     \right) \;  \bigcup\limits_{\phi_{p_i} }\;
     \Omega,
\end{equation*}
where the two sets are glued by $\phi_i$ along a suitable neighborhood of $p_{i}$.

The boundary $\pa\Sigma(\Omega)$ can be identified with the union of
$\pa'\Sigma(\Omega)$ and $\pa^{''}\Sigma(\Omega)$, where $\pa'\Sigma(\Omega)$  is the union of hyperfaces that are not at infinity,
and $\pa^{''}\Sigma(\Omega)$ is the union of hyperfaces that are at infinity.  In the terminology of \cite{ALN}, the hyperfaces
 {\em not at infinity}
  correspond to actual faces of $\Omega$, while
  hyperfaces {\em at infinity}
correspond to a singularity of $\Omega$.
We can write
\begin{equation}\label{bdry.desing}
  \pa \Sigma(\Omega) \cong \left(\bigsqcup\limits_{p_i \in
    \Omega^{(0)}}[0,1)\times \partial\omega_{p_i} \cup
    \{0\}\times \overline{\omega_{p_i} } \right)
    \bigcup\limits_{\phi_{p_i}, \, p_i\in\Omega^{(0)}} \Omega_0.
\end{equation}
where $\Omega_0$ denotes the smooth part of $\partial\Omega$, that is,
$\Omega_0:=\partial \Omega \backslash \Omega^{(0)}$.
We denote by
\begin{equation}\label{M=bdry.desing}
M:= \pa'\Sigma(\Omega)  \cong \left(\bigsqcup\limits_{p_i \in
    \Omega^{(0)}}[0,1)\times \partial\omega_{p_i}  \right)
    \bigcup\limits_{\phi_{p_i}, \, p_i\in\Omega^{(0)}} \Omega_0.
    \end{equation}
Note that $M$ coincides with
the closure of $\Omega_0$ in $\Sigma(\Omega)$. It is a compact manifold with (smooth) boundary
 $$\pa M= \bigsqcup\limits_{p_i \in
    \Omega^{(0)}}\{0\} \times \partial\omega_{p_i}.$$
In fact, we regard $M:=\partial' \Sigma
(\Omega)$ as a desingularization of the boundary $\pa \Omega$. Operators on $M$ will be related to (weighted) operators on $\pa \Omega$, as we shall see in Section 4.
See  \cite{BMNZ, CQ13} for more details.

%%%%%%%%%%%%%%%%%%%%%

\smallskip
\subsection{Groupoid construction for conical domains without cracks}\label{ss.groupoidnocrack}

Let $\Omega$ be a conical domain without cracks, $\Omega^{(0)}=\{ p_1, p_2, \cdots, p_l\}$
be the set of (true) conical points of $\Omega$, and $\Omega_0$ be the smooth part of $\pa\Omega$. We will review the definition of the \emph{layer potentials groupoid} $\maG \tto M$, with $M:=\pa'\Sigma(\Omega) $  a compact set, as in the previous subsection, following \cite{CQ13}.

\smallskip

Let $\maH:=[0,\infty) \rtimes (0,\infty) $ be the transformation groupoid
with the action of $(0,\infty)$ on $[0,\infty)$ by dilation [Example \ref{transformation}].
To each $p_i \in \Omega^{(0)}$,
we first associate a groupoid $\maH \times (\pa\omega_{p_i})^2 \tto [0,\infty) \times \pa\omega_{p_i}  $,  where  $(\pa\omega_{p_i})^2$ is  the pair groupoid of $\pa\omega_{p_i}$ [Example \ref{pair_gpd}].
 We then take its reduction to $[0,1) \times \pa\omega_{p_i}$ to define
 $$\maJ_{i}:= \left( \maH \times (\pa\omega_{p_i})^2\right)_{[0,1) \times \pa\omega_{p_i}}^{[0,1) \times \pa\omega_{p_i}}\tto [0,1) \times \pa\omega_{p_i}. $$
Let $\Omega_0^2$ be the pair groupoid of $\Omega_0$. We now  want to glue $\Omega_0^2$ and $\maJ_i$ $(i=1,2,\cdots, l)$
in a suitable way. In fact, over the interior, we have
$\maJ_i|_{(0,1)\times \pa\omega_{p_i}} \simeq (0,1)^2 \times (\pa\omega_{p_i})^2$,
the Cartesian product of two pair groupoids.
We can take a suitable neighborhood $V_i \subset
\mathbb{R}^n$ of $p_i$ and define a diffeomorphism $V_{i}\cap\Omega_{0} \cong (0,\varepsilon_{i})\times \pa\omega_{i} \cong (0,1)\times \pa\omega_{i}$
which leads to a map $\varphi_{i}: int \left(\maJ_{p_i}\right)
\rightarrow \Omega_0^2$
such  that $\varphi_{i}$ is smooth, a
diffeomorphism into its image, and preserves the groupoid structure of
$\maJ_i$ and $\Omega_0^2$. Let $\varphi=(\varphi_{i})_{p_i\in \Omega^{(0)}}$ on the disjoint union. The following definition then makes sense.

\begin{definition}\label{gpd1}
Let $\Omega$ be a conical domain without cracks. The \emph{layer potentials groupoid} associated to $\Omega$ is the Lie groupoid $\maG \tto M:=\pa'\Sigma(\Omega)$ defined
by
\begin{equation}\label{grpd.nocrack}
  \cG:=\left(\bigsqcup\limits_{p_i \in \Omega^{(0)}}\maJ_{p_i} \right)\quad
  \bigcup\limits_{\varphi} \quad
 \Omega_0^2 \quad  \tto \quad M
\end{equation}
where $\varphi=(\varphi_{p_i})_{p_i\in \Omega^{(0)}}$, with space of units
\begin{eqnarray}\label{units.nocrack}
  M &=& \left(\bigsqcup\limits_{p_i \in
    \Omega^{(0)}}[0,1)\times \partial\omega_{p_i} \right)\quad
    \bigcup\limits_{\varphi} \quad \Omega_0 \quad  \cong \quad\partial' \Sigma (\Omega),
\end{eqnarray}
where $\partial' \Sigma (\Omega)$ denotes the union of hyperfaces
which are not at infinity of a desingularization.
\end{definition}

Clearly, the space $M$ of units is
compact. We have that
$\Omega_0$ coincides with  the interior of $M$, so $\Omega_0$ is an open dense subset of $M$.
The following proposition summarizes the properties of the layer potentials groupoid and its groupoid $C^{*}$-algebra. Note that  $C^*(\cH)=\maC_0([0,\infty))\rtimes \RR^+$,
by   \cite{MRen}.

\begin{proposition}\label{Groupoid}
Let $\cG$ be the layer potentials groupoid \eqref{grpd.nocrack} associated to a domain
with conical points $\Omega\subset \RR^n$. Let
$\Omega^{(0)}=\{p_1,p_2,\cdots, p_l\}$ be the set of conical points
and $\Omega_0=\partial \Omega \backslash \Omega^{(0)}$ be the smooth
part of $\partial\Omega$.  Then, $\cG$ is a Lie groupoid with units
$M= \partial' \Sigma (\Omega)$ such that
\begin{enumerate}

\item   $ \Omega_0$ is an open, dense invariant subset with
$\cG_{\Omega_0} \cong \Omega_0 \times \Omega_0$ and $\Psi^m(\cG_{\Omega_0})\cong \Psi^m(\Omega_0)$.

\item For each conical point $p \in \Omega^{(0)}$, $\{ p \} \times \pa
\omega_p$ and $\pa  M= \bigcup\limits_{p \in \Omega^{(0)}} \{ p \} \times \pa
\omega_p$ are invariant subsets and
\begin{equation*}
\cG_{\pa M}= \bigsqcup\limits_{i=1}^{l} (\pa\omega_i \times \pa\omega_i) \times (\RR^+ \times \{p_i\})
\end{equation*}
  \item If $P\in \Psi^m(\cG_{\pa M})$ then for each $p_i\in \Omega^{(0)}$, $P$ defines a Mellin convolution operator on $\RR^+\times \pa \omega_i$.
\smallskip

\item $\cG$ is (metrically) amenable, i.e., $C^*(\cG)\cong C^*_r(\cG)$.

\item If $n\geq 3$, $C^*(\cG_{\pa M})\cong  \bigoplus\limits_{i=1}^{l} \maC_0(\RR^+) \otimes \maK$. If $n=2$, $C^*(\cG_{\pa M})\cong \bigoplus\limits_{i=1}^{l} M_{k_i}(\maC_0(\mathbb{R}^+)) $,
where $k_i$ is the number of elements of $\pa \omega_i$ and
  $l$ is the number of conical points.
\end{enumerate}

\end{proposition}
Note that if $P\in \Psi^{m}(\maG)$ then,
at
the boundary, the regular representation yields an operator
$$P_{i}:= \pi_{p_i}(P) \in \Psi^m(\RR^+\times (\pa \omega_{i})^2),$$
which is
defined by a distribution kernel $\kappa_i$ in $\RR^+\times (\pa
\omega_{i})^2$, hence a Mellin
convolution operator on $\RR^+\times \pa \omega_i$ with kernel
$\tilde{\kappa_{i}}(r,s, x',y'):=\kappa_i(r/s, x', y')$.
 If  $P\in \Psi^{-\infty}(\maG)$, that is, if $\kappa_i$ is smooth, then it defines a smoothing
{Mellin convolution operator} on $\RR^+\times \pa \omega_{i}$ (see
\cite{LP,QN12}). This is  one of the motivations in our
definition of $\cG$.

\begin{remark}\label{rmk.bgrpd}
Recall the definition of $b$-groupoid in Example~\ref{bgrpd}, which,
in the case of $M=\bigsqcup_{i} [0,1) \times \pa \omega_{i}$ comes down to
$$
  {^b\cG}=  \bigsqcup_{i,j} \;\RR^+\times (\pa_j \omega_{i})^2 \;\; \bigcup \;\; \Omega_0^2 \;
$$
where $\pa_j \omega_{i}$ denote the connected components of $\pa \omega_{i}$.
If $\pa \omega_{i}$ is connected,  for all $i=1,...,l$, then $\cG={^b\cG}$. In many cases of interest, $\pa \omega$ is not
connected, for instance, if $n=2$, that is,
if we have a polygonal domain, then $\pa \omega$ is  {always}
disconnected.
 In general, the groupoid
$\cG$ is larger and not $d$-connected, and ${^b\cG}$ is an open, wide
subgroupoid of $\cG$. (The main difference is that here we allow the different
connected components of the boundary, corresponding to the same conical point, to interact, in that there are arrows between
them.)
The Lie algebroids of these two groupoids coincide, as $A(\cG)\cong {^bTM}$, the $b$-tangent bundle of $M$.
 Moreover,  $\Psi(\cG)\supset \Psi({^b\cG})$, the (compactly supported) $b$-pseudodifferential operators on $M$.
\end{remark}

%%%%%%%%%%%%%%%%%%%%%%%%%%
\smallskip
\smallskip

\subsection{Groupoid construction for polygonal domains with ramified cracks}\label{ss.groupoidcracks}
Let us first recall the definition of polygonal domains with ramified cracks from \cite{CQ13}. In this subsection, $n=2$.

\begin{definition}\label{def.crack}
Let $\Omega\subset \RR^2$ be a polygonal domain as in Definition \ref{domain}. Then $\Omega$ is a \emph{domain with cracks if $\pa \Omega \neq \pa \overline{\Omega}$.} Let $x \in \pa\Omega$. Suppose that there  exists a neighborhood $V_x$ of $x$ in $\Omega$, an open subset $\omega_x \subset S^1$, $\omega_x \neq S^1$, and a diffeomorphism $\phi_x: V_x \rightarrow B^2$ such that in polar coordinates $(r,\theta)$, we have
$$\phi_x(V_x \cap \Omega)= \{(r,\theta), r\in (0,1), \theta \in \omega_x\}.$$
Then $x$ is a {\em crack point} if $\pa \omega_x \neq \pa \overline{\omega_x}$. Moreover:
\begin{enumerate}
\item  $x$ is an {\em inner crack point} if $\overline{\omega_x}= S^1$ (that is, if $x\in int(\overline{\Omega})$).

\item $x$ is an {\em outer crack point} if $\overline{\omega_x} = S^1_+$ (that is, if $x\in( \pa(\overline{\Omega}))_{0} $, smooth part of the boundary of $\overline{\Omega}$).
\item otherwise, $x$ is a {\em conical crack point}.
\end{enumerate}
If $x$ is in the interior of the set  of all crack points, then  we call $x$ a \emph{smooth crack point}; in this case, $\omega_x\cong S^1_+\sqcup S^1_+$, 
where $S^1_+$ is the hemisphere. Otherwise, we say that $x$ is a \emph{singular crack point.}
\end{definition}

\begin{remark}
According to our terminology, we classify all crack points by two different ways:
\begin{enumerate}
\item[i)] inner crack point, outer crack point, or conical crack point;
\item[ii)] smooth crack point, or singular crack point.
\end{enumerate}
Smooth crack points correspond to the interior of crack curves and are always inner. The endpoints are singular, and may be inner, outer or conical crack points.
Note that the number of singular crack points is finite. We will use the number of connected components of $\omega_x$
to define the {\em ramification number} for each crack point (see below).
\end{remark}

Let $\Omega$ be a polygonal domain with cracks, and $\Omega_0$ be the smooth part of $\pa\Omega$.
Following the work in \cite{CQ13} (see also \cite{LiMN, MN10}), we define the {\em unfolded boundary} $\pa^u \Omega$ to be
the set of inward pointing unit normal vectors to $\Omega_{0}$.
Then the {\em unfolded domain} is defined to be
$$\Omega^u=\Omega \cup \pa^u\Omega.$$
The main idea is that a smooth crack point $x$ should be covered by two points,
which correspond to the two sides of the crack (and the two possible
non-tangential limits at $x$).
%of functions defined on $\Omega$.
Hence $\Omega^u$ is a (generalized) polygonal domain without cracks, i.e., a conical domain without cracks.

Then, let us specify the boundary of $\Omega^u$. It is easy to see that  the smooth part of $\pa\Omega^u$
is just $\pa^u\Omega$,  that is, the union of the smooth part of $\pa\Omega$
and the $2$-covers of the smooth crack curves, and that the non-crack true vertices are also vertices in $\pa \Omega^u$.
As for crack points, first of all, we associate a {\em a ramification number} to each crack point as follows.
\begin{enumerate}
\item If $x$ is an inner or outer crack point, the ramification number $k_{x}$ associated to $x$ is defined
 to be the number of connected components of $\omega_x$.
 \item If $x$ is conical crack point, then we decompose $$\omega_x=\omega_x' \cup \omega_x'',$$ where
 $\overline{\omega_x'}$ and $\overline{\omega_x''}$ are disjoint
  such that $\pa\omega_x'= \pa\overline{\omega_x'}$ and
 $\pa\omega_x'' \neq \pa\overline{\omega_x''}$, and $\omega_x'$ is maximal satisfying these conditions.
 Since $\omega_x''$ is non-empty, denote by $k$ the number of connected components of $\omega_x''$.
 The ramification number associated to $x$ is then defined to be $k${ if $\omega_x'=\emptyset$, and $k+1$ otherwise.}
\end{enumerate}
The ramification number gives us the number of different ways we can approach the boundary close to $x$. Smooth crack points always have ramification number $2$.
For instance, a point
with ramification number $1$ is an end point of some crack curve, hence is a singular crack point.

According to our terminology, the boundary $\pa\Omega^u$ of $\Omega^u$ consists of $\pa^u \Omega$, the non-crack vertices
of $\Omega$, and $k$-covers of singular crack points with ramification number $k$.
\smallskip

Let $\maC=\{c_1, \cdots, c_m\}$ be
the set of singular crack points with $c_1. \cdots, c_{m'}$ the conical crack points with non-empty non-crack part, $\omega_{c_j}'\neq \emptyset$.
Then we define
$$\maC^u:= \{c_{ji} \, |\, c_{ji} \,\,\text{covers} \,\, c_j, i=1,2,\cdots, k_{c_j}\}\; \subset \;  \pa \Omega^{u},$$
where $k_j$ is the number of connected components of $\omega_{c_j}$ with $c_j$ an inner or outer crack point,
or the number of connected components of the crack part $\omega_{c_j}''$ with $c_j$ a conical point.
Moreover, if $j=1,.., m'$ and $\omega_{c_j}' \neq \emptyset$, the cover over $c_j$ is considered together with
 a point $c_{j0}:=c_{j}$, representing the non-crack part.
 Therefore, the set of vertices $V^u$ of $\Omega^u$ can be
decomposed as
$$V^u= \Omega^{(0)} \cup \{ c_{j} \}_{j=1, \cdots m'} \cup \maC^u,$$
where $\Omega^{(0)}=\{p_1,\cdots, p_l\}$ is now the set of non-crack conical points.

For $x\in V^u$, let $\omega_x^u$ the base of the cone at $x$ in $\Omega^u$. If $x\in \Omega^{(0)}$, then $\omega_x^u=\omega_{x}$ and if $x=c_{j}, j=1,..., m'$, then $\omega_x^u=\omega'_{x}$, the non-crack part.  If $c_{j} \in\cC\subset \pa\Omega $ is a singular outer or inner crack point, then the
base of the cone  (in $\Omega$) is a union of open intervals
$$\omega_{c_{j}} = \cup_{k=1}^{k_{c_{j}}} I_{c_{j}k}, \quad  I_{c_{j}k}= ] \theta_{j,k-1}, \theta_{j,k}[,$$
with $k_{c_{j}}$ the ramification number  (if $c_{j}$ is inner then $\theta_{0}=0$, $\theta_{ k_{c_{j}} }=2\pi$,  if $c_{j}$ is outer then $\theta_{0}=0$, $\theta_{ k_{c_{j}} }=\pi$).
In particular, for each $x=c_{jk}\in \cC^u\subset V^u$ in the cover of $c_{j}$, we have
$\omega^u_{c_{jk}}= I_{c_{j}k}= ] \theta_{j,k-1}, \theta_{j,k}[$.
If $c_{j}$ is a conical crack point, then we  replace $\omega_{c_{j}}$ by the crack part $\omega''_{c_{j}}$. 

 So we are able to apply the construction of the previous subsection to $\Omega^u$
to obtain a Lie groupoid.
As before, let $\maH:=\left( [0, \infty) \rtimes (0,\infty) \right)_{[0,1)}^{[0,1)}\tto [0,1)$. Then the groupoid associated to the generalized conical domain $\Omega^u$, with no cracks,  should be defined, according to the construction in the previous subsection:

\begin{eqnarray*}
  \cG^u %& := & \left( \bigsqcup\limits_{x\in V^u}\maH\times (\partial\omega_{x}^u)^2\right) \quad \bigcup\limits_{\varphi} \quad (\pa^u \Omega)^2\\
  & := & \left(
  \bigsqcup\limits_{p_i \in \Omega^{(0)}}\maH\times (\partial\omega_{p_i})^2
  \quad \bigcup \quad
     \bigsqcup\limits_{ j=1}^{m' }\maH\times (\partial\omega'_{c_j})^2
   \quad \bigcup
   \bigsqcup\limits_{c_{jk}\in\maC^u}  \maH \times (\partial I_{{c_j}k})^2
  \right) \quad
     \quad \bigcup\limits_{\varphi} \quad (\pa^u \Omega)^2,
\end{eqnarray*}
where $ I_{{c_j}k}$ is the $k$-th connected component of $\omega_{c_j}$, respectively, of $\omega''_{c_j}$, if $c_j$ is non-conical, respectively, $c_j$ is a conical crack point, and $\varphi=(\varphi_{x})_{x\in V^u}$.

Noting that $\pa \omega_x$ is a discrete set, with $\omega_x^{u}$ an interval on covers of crack points, and denoting by $\maP_k$ the pair groupoid of a discrete set with $k$ elements, we get to the following definition  (writing $(\bigsqcup A)^\alpha$ for the disjoint union of $\alpha$ copies of $A$):

\begin{definition}\label{gpd2}
Let $\Omega$ be a polygonal domain with ramified cracks, $\Omega^{(0)}=\{p_1,\cdots, p_l\}$ be the set of non-crack conical points, $\maC=\{c_1, \cdots, c_m\}$ be
the set of singular crack points with $c_1. \cdots, c_{m'}$ the conical crack points with non-empty non-crack part. The Lie groupoid $\maG^u \tto M^u:=\pa'\Sigma(\Omega^u)$
called the {\em layer potential groupoid associated to $\Omega$} is defined by
\begin{equation}
\label{grpd.crack}
 \cG^u = \left(
  \bigsqcup\limits_{p_i \in \Omega^{(0)}}\cH\times \maP_{2k_{p_i}}    \; \bigcup \;
    \bigsqcup\limits_{ j=1, \cdots, m' }\cH\times \maP_{2k_{c_j}'} \; \bigcup \;
  \left(\bigsqcup\limits  \cH \times \maP_{2} \right)^\alpha \right) \;  \bigcup\limits_{\varphi} \; (\pa^u \Omega)^2,
\end{equation}
where $k_{p_i}$, $k_{c_j}'$ are the number of connected components of $\omega_{p_i}$ and $\omega_{c_j}'$, respectively, and $\alpha:= k_{c_1} + \cdots + k_{c_m}{-m'}$ is the total ramification number of $\Omega$, $k_{c_{j}}$ the ramification number of $c_{j}$. The space of units of $\maG^u$ is given by
\begin{eqnarray*}\label{units.crack}
&&\\
  M^u &:=& \left(\bigsqcup\limits_{p_i \in
    \Omega^{(0)}} [0,1)\times \partial\omega_{p_i}
      \;\bigcup \;
    \bigsqcup\limits_{ j=1}^{m' }\, [0,1)\times \partial\omega'_{c_j}
    \; \bigcup \;
    \bigsqcup\limits_{c_{ji}\in\maC^u}  [0,1)\times  \partial I_{{c_j}k}
    \right)\;
    \bigcup\limits_{\varphi}\;  \pa^u \Omega \nonumber.\\
  \end{eqnarray*}

\end{definition}

From the above definition, we have
\begin{eqnarray}\label{boundary2}
\label{ }
\pa M^u = \bigcup\limits_{x \in V^u} \{ x \} \times \pa \omega_x^u =
\bigcup\limits_{p \in \Omega^{(0)}} \{ p \} \times \pa \omega_p \;
\bigcup\limits_{j=1}^{m'} \; \{ c_j \} \times \pa \omega_{c_j}'\;
\bigcup\limits_{c\in \maC^u} \; \{ c \} \times \pa \omega_c^u \;
\end{eqnarray}
where $\omega_x^u$ is the base of the cone at $x$ in $\Omega^u$.
We remark that the unfolded boundary $\pa^u\Omega$ is an open, dense, $\maG^u$-invariant set of $M^u$.

\vspace{0.3cm}
\section{Fredholm Conditions for Operators on Layer Potentials Groupoids}\label{s.FredCond}

In this section, we will adapt layer potential groupoids for conical domains
constructed in Section \ref{s.LP_groupoids} to the
framework of Fredholm groupoids, and then obtain the Fredholm criterion
for operators on layer potential groupoids.

\smallskip
\smallskip

\subsection{Desingularization and weighted Sobolev spaces for conical domains}

An important class of function spaces on singular manifolds are weighted Sobolev spaces.
Let $\Omega$ be a conical domain, and $r_\Omega$ be the smoothened distant function to the set of conical points  $\Omega^{0}$
as in \cite{BMNZ, CQ13}.
The space $L^2(\Sigma(\Omega))$ is defined using the volume element of a compatible metric
on $\Sigma(\Omega)$. A natural choice of compatible metrics is $g= r^{-2}_\Omega \, g_e$, where
$g_e$ is the Euclidean metric. Then the Sobolev spaces $H^m(\Sigma(\Omega))$ are defined  in the usual way, with
pivot $L^2(\Sigma(\Omega))$. These Sobolev spaces can be identified with weighted Sobolev spaces.

\smallskip
Let $m\in \ZZ_{\geqslant 0}$, $\alpha$ be a multi-index. The $m$-th
Sobolev space on $\Omega$ with weight $r_{\Omega}$ and index $a$ is
defined by
\begin{equation}\label{def.weighted}
  \maK_{a}^m(\Omega)=\{u\in L^2_{\text{loc}}(\Omega) \, | \,\,
  r_{\Omega}^{|\alpha|-a}\partial^\alpha u\in L^2(\Omega), \,\,\,\text{for
    all}\,\,\, |\alpha|\leq m\}.
\end{equation}
We defined similarly the spaces $ \maK_{a}^m(\pa\Omega)$. Note that in this case, as $\pa\Omega$ has no boundary, these spaces are defined for any $m\in \mathbb{R}$ {by complex interpolation \cite{BMNZ}}.

The following result is taken from \cite[Proposition~5.7 and
Definition 5.8]{BMNZ}.

\begin{proposition} \label{Identification}
Let $\Omega \subset \RR^n$ be a domain with conical points,
$\Sigma(\Omega)$ be its desingularization, and $\partial'\Sigma(\Omega)$
be the union of the hyperfaces that are not at infinity.
We have
\begin{enumerate}
\item[(a)] $ \maK^{m}_{\frac{n}{2}}(\Omega)\simeq H^{m}(\Sigma(\Omega),
  g),$ for all $m\in \mathbb{Z}$;
\item[(b)] $ \maK^{m}_{\frac{n-1}{2}}(\partial\Omega)\simeq
  H^{m}(\partial'\Sigma(\Omega),g)$, for all $m\in \mathbb{R}$.
\end{enumerate}
where the metric $g= r^{-2}_\Omega \, g_e$ with $g_e$ the Euclidean metric.
\end{proposition}

\subsection{Freholm criteria for operators on layer potentials groupoids}

To get Fredholm criteria on layer potentials groupoids, we start with checking that  these groupoids are indeed Fredholm.

Let us first see the case of straight cones. Let $\omega \subset S^{n-1}$ be an open subset with smooth
boundary (note that we allow $\omega$ to be \emph{disconnected}) and

$$\Omega := \{t y',\ y' \in \omega,\ t \in (0, \infty)\} =  \mathbb{R}^{+} \,\omega$$
be the (open, unbounded) cone with base $\omega$. The desingularization becomes  in this case  an
half-infinite solid cylinder
\begin{equation*}
\Sigma(\Omega)=  [0,\infty) \times \overline{\omega}
\end{equation*}
with boundary $\pa \Sigma(\Omega) = [0,\infty) \times \partial \omega
\cup \{0\}\times \omega$, so that  $M= \partial'
\Sigma(\Omega)=[0,\infty)\times \partial \omega$ the union of the
hyperfaces not at infinity. Taking the one-point compactification of $[0,\infty]$, we can consider the groupoid $\overline{\maH}$ as in Example \ref{expl.transfgroupoidFredholm}, and
 the {layer potentials groupoid associated to a straight cone
  $\Omega\cong \mathbb{R}^+\omega$} is the product Lie groupoid with units
$M=[0,\infty] \times \pa \omega$, corresponding to a desingularization
  of $\pa \Omega$, defined as
\begin{equation*}\label{cJ}
\cJ := \overline{\maH} \times (\partial\omega)^2.
\end{equation*}
Now, we have seen in Example \ref{expl.transfgroupoidFredholm} that $\overline{\maH}$ is a Fredholm groupoid, hence $\maJ$ is also a Fredholm groupoid.

In the general case, we can proceed in several ways: we can use the same argument as in the straight cone case (that is, as in Example \ref{expl.transfgroupoidFredholm} ),  or we can use the fact that the gluing (along the interior) of Fredholm groupoids is also a Fredholm groupoid. By analogy with the classes of Fredholm groupoids studied in \cite{CNQ17}, we chose to check that $\maG$ is actually given by a fibered pair groupoid over the boundary, that is, a strong submersion groupoid..

\begin{theorem}\label{thm.LPFred}
The layer potentials groupoids defined in Definition \ref{gpd1} and Definition \ref{gpd2} are Fredholm groupoids.
\end{theorem}

\begin{proof}
Let us deal with the case of conical domains without cracks. The other case is similar, taking the unfolded boundary.

It is clear that $\Omega_0$ is an open, dense, $\maG$-invariant subset of $M= \pa'\Sigma(\Omega)$, with $\maG_{\Omega_{0}}$ is the pair groupoid. Let
$$F:= M\backslash \Omega_0 =\pa M = \bigcup\limits_{p \in \Omega^{(0)}} \{ p \} \times \pa \omega_{p}\cong  \bigsqcup\limits_{i=1}^{l} \pa \omega_{p_{i}}.$$
We have
\begin{equation*}
\cG_{F}= \bigsqcup\limits_{i=1}^{l} (\pa\omega_i \times \pa\omega_i) \times (\RR^+ )
\end{equation*}
For any $x\in F$, we have $(\maG_F)_x^x =\maG_x^x \simeq \{ x\} \times \mathbb{R}^+ \simeq \mathbb{R}^+ $.
Since the group $\mathbb{R}^+ $ is commutative, it is amenable.
We claim that $\maR(\maG_{F})=\{\pi_x,\, x \in F\}$ is a strictly spectral / exhaustive set of
  representations of $C\sp{\ast}(\maG_F)$.
This can be proved directly, using the description in (4) of Proposition \ref{Groupoid}.

We show alternatively that $\maG_{F}$ can be given as a fibered pair groupoid, along the lines of Example \ref{ex.help-for-lp}.
Let $\maP:= \left\{  \pa\omega_i  \right\}_{i=1,..., l}$ be a finite partition of the smooth manifold $F$ and let $f: F\to \maP$, $x\in  \pa\omega_i  \mapsto  \pa\omega_i  $.
Then $ \pa\omega_i $ are closed submanfolds of $F$ and $\maP$ is a smooth discrete manifold, with $f$ is a locally constant smooth fibration.

Let $\maH:= \maP\times \RR^{+}$, as a product of a manifold and a Lie group. Then, by Example \ref{ex.help-for-lp},
$$f\pullback(\maH) = \bigsqcup\limits_{i=1}^{l} (\pa\omega_i \times \pa\omega_i) \times \RR^+  =  \maG_{F}.$$
Hence, by Corollary \ref{prop.fred}, the result is proved.

\end{proof}

If we apply Theorem \ref{thm.nonclassical2} (\cite[Theorem~4.17]{CNQ}) to our cases, we obtain the main theorems as follows. Recall that the regular representations $\pi_{x}$ and $\pi_{y}$ are unitarily equivalent for $x,y$ in the same orbit of $\maG_{F}$, so that for $P\in \Psi^{m}(\maG)$  we obtain a family of Mellin convolution operators
$P_{i}:=\pi_{x}(P)$ on $\RR^+\times \pa\omega_i$, $i=1,...,p$, with
$x=(p_{i},x')   \in \pa M$, $x'\in \pa\omega_{p_{i}}$.

Recall that the space $L^{m}_{s}(\maG)$
is the {\em norm closure} of $\Psi^m(\maG)$ in the topology of
continuous operators $H^s(M )\to H^{s-m}(M)$. By the results in \cite{QL18, QN12}, if $P\in  L^{m}_{s}(\maG)$, then $\pi_{p_{i}}(P)$ is also a Mellin convolution operator.

\begin{theorem}\label{thm.fredholm}
Suppose that $\Omega \subset \mathbb{R}^n$ is a conical domain without cracks and
$\Omega^{(0)}=\{p_1,p_2,\cdots, p_l\}$ is the set of conical points. Let $\maG\tto M=\pa'\Sigma(\Omega)$
be the layer potentials groupoid as in Definition \ref{gpd1}.
 Let $P\in L^{m}_{s}(\maG)\supset \Psi^m(\maG)$ and $s \in \RR $.
We have
\begin{equation*}
  \begin{gathered}
	P : \maK_{\frac{n-1}{2}}^s(\pa\Omega) \to \maK_{\frac{n-1}{2}}^{s-m}(\pa\Omega) \mbox{ is Fredholm}
        \ \ \Leftrightarrow \ \ P \mbox{ is elliptic and the Mellin convolution operators  }\\
	\ P_{i} : H^s(\RR^+\times \pa\omega_i;g) \to H^{s-m}(\RR^+\times \pa\omega_i;g)\,,
        \, i=1,...,p, \mbox{  are invertible}\,,
  \end{gathered}
\end{equation*}
where the metric $g= r^{-2}_\Omega \, g_e$ with $g_e$ the Euclidean metric.
\end{theorem}

As for the case of conical domains with cracks, recall that the boundary $\pa M^u$
of the space of units $M^u$ consists of three different types of points in Equation (\ref{boundary2}).
Likewise, we obtain the following theorem.

\begin{theorem}\label{thm.crack}
Suppose that $\Omega \subset \mathbb{R}^2$ is a conical domain with cracks, $\Omega^{(0)}=\{p_1,\cdots, p_l\}$ is the set of non-crack conical points, $\maC=\{c_1, \cdots, c_m\}$ 
the set of singular crack points with $c_1. \cdots, c_{m'}$ the conical crack points with non-empty non-crack part. Let $\maG^u \tto M ^u=\pa'\Sigma(\Omega^u)$
be the layer potentials groupoid as in Definition \ref{gpd2}. 
Let $P\in L^{m}_{s}(\maG)\supset \Psi^m(\maG^u)$ and $s \in \RR$.
We have	
\begin{equation*}
  \begin{gathered}
	P : \maK_{\frac{n-1}{2}}^s(\pa\Omega) \to \maK_{\frac{n-1}{2}}^{s-m}(\pa\Omega) \mbox{ is Fredholm}
        \ \ \Leftrightarrow \ \ P \mbox{ is elliptic and the Mellin convolution operators   }\\
   \end{gathered}
\end{equation*}
\begin{enumerate}
  \item  $P_i : H^{s}(\RR^+ \times \pa \omega_{p_{i}}, g)\rightarrow H^{s-m}(\RR^+ \times \pa \omega_{p_{i}},g)$, ,$ i=1,...,p$,
  \item $P'_j : H^{s}(\RR^+ \times \pa \omega'_{c_j}, g)\rightarrow H^{s-m}(\RR^+ \times \pa \omega'_{c_j},g)$,  $j=1,  \cdots, m'$;
  \item $P_{jk} : H^{s}(\RR^+ \times \pa I_{c_jk}, g)\rightarrow H^{s-m}(\RR^+ \times \pa I^h_{c_j},g)$
  with 
  $j=1, \cdots, m$, $k=1, \cdots k_{c_j}$ and $ I_{c_jk}$ is the $k$-th connected component of either $\omega_{c_j}$ if $c_j$ is non-conical, or
  of $\omega''_{c_j}$ if $c_j$ is a conical crack point,
\end{enumerate}
are invertible.
\end{theorem}

As an application of the above theorems, we can show that certain boundary integral operators are invertible on appropriate weighted
Sobolev spaces. Let us briefly recall some result in \cite{QL18}. Let $\Omega$ be a simply connected polygon in $\mathbb{R}^2$ with (interior)
angles $\theta_1, \theta_2, \cdots, \theta_n$.
The Neumann-Poincar\'{e} (NP, or double layer potential) operator $K$ is the integral operator
$$Kf(x)=-\dfrac{1}{\pi} \int_{\pa\Omega} \dfrac{(x-y)\cdot\nu(y)}{|x-y|^2} f(y) \, dS(y),$$
where $dS$ is the induced measure on $\pa\Omega$, $\nu(y)$ is the outer normal unit vector to the boundary at $y$.
Define $$\theta_0 := \min\{\frac{\pi}{\theta_1}, \frac{\pi}{2\pi-\theta_1},\frac{\pi}{\theta_2}, \frac{\pi}{2\pi-\theta_2}, \cdots, \frac{\pi}{\theta_n},\frac{\pi}{2\pi-\theta_n}\}.$$
By analysing carefully the behavior of $I + K$ near each vertex, we are able to show the invertibility of $I +K$ at each vertex. Hence Theorem \ref{thm.fredholm} becomes applicable.
We obtain
\begin{theorem}
 For $a\in (-\theta_0, 1/2)$ and $m\geq 0$, the operators
$$\pm I+K: \mathcal{K}^m_{\frac{1}{2}+a}(\partial\Omega)\rightarrow \mathcal{K}^m_{\frac{1}{2}+a}(\partial\Omega) $$
are both Fredholm.
\end{theorem}
Combing some results from PDE's, we are able to show that the operators $\pm I+K $ are actually isomorphisms, which implies a solvability result in weighted Sobolev spaces for the Dirichlet problem for Laplace's equation on $\Omega$.

\vspace{0.3cm}
\bibliographystyle{plain}
%\bibliography{bib,nistor}
% \bibliography{../../bib,../../nistor}

\def\cprime{$'$} \def\cprime{$'$} \def\cprime{$'$}

\end{document}